\documentclass[10pt]{amsart}
\usepackage{enumerate,amsmath,amssymb,latexsym,
amsfonts, amsthm, amscd, MnSymbol}

\theoremstyle{plain}

\newtheorem{theorem}{Theorem}

\numberwithin{equation}{section}

\newcommand{\ra}{\rightarrow}

\newcommand{\LL}{\mathbb{L}}
\newcommand{\MM}{\mathbb{M}}

\addtocounter{section}{-1}

\addtocounter{theorem}{-1}

\begin{document}

\title {A Gr\"atzer-Schmidt theorem for the Lindenbaum-Tarski algebra of IPC}

\date{}

\author[P.L. Robinson]{P.L. Robinson}

\address{Department of Mathematics \\ University of Florida \\ Gainesville FL 32611  USA }

\email[]{paulr@ufl.edu}

\subjclass{} \keywords{}

\begin{abstract}

We prove a version of the Gr\"atzer-Schmidt theorem for the Lindenbaum-Tarski algebra associated to the Implicational Propositional Calculus. 

\end{abstract}

\maketitle

\medbreak

\section{Introduction} 

\medbreak 

Gr\"atzer and Schmidt established that the skeleton of a pseudocomplemented semilattice is naturally Boolean; see [1] for the details as they apply to a meet-semilattice with zero. For our purposes, it is convenient to express the theorem in its dual form: thus, let $\MM$ be a topped join-semilattice and write the (dual) pseudocomplement of $a \in \MM$ as $a^*$ so that $a^* \vee a = 1$ and if $x \in \MM$ satisfies $x \vee a = 1$ then $a^* \leqslant x$; the (dual) skeleton $S(\MM) = \{a^* : a \in \MM\}$ of $\MM$ is then a Boolean lattice for the induced partial order, with pairwise supremum given by $\vee$ and with $\inf \{a, b \} = (a^* \vee b^*)^*$ whenever $a, b \in S(\MM)$. The Implicational Propositional Calculus IPC, having the conditional $\supset$ as its only connective and incorporating the Peirce axiom scheme, has an associated Lindenbaum-Tarski algebra obtained by the identification of (syntactically) equivalent well-formed formulas. This algebra is naturally a join-semilattice by virtue of the Peirce axiom scheme and is topped by the equivalence class comprising all theorems. Now, let $Q$ be a fixed IPC formula and $q$ its equivalence class; for the equivalence class $z = [Z]$ of any IPC formula we write $z^q := [Z \supset Q]$. This well-defines an order-reversing unary operation $(\bullet)^q$ on the Lindenbaum-Tarski algebra: it satisfies the requirement $z^q \vee z = 1$ of a pseudocomplement (indeed, this amounts to a restatement of the Peirce scheme); but it does not satisfy the companion requirement that if $w \vee z = {\bf 1}$ then $z^q \leqslant w$. Nevertheless, we prove in Theorem \ref{main} that for each $Q$ the corresponding `skeleton' (comprising all $z^q$ as $z$ runs over the Lindenbaum-Tarski algebra) is still naturally a Boolean lattice. As might be expected, some resistance is offered by distributivity, for which we provide two proofs. 

\medbreak 

\section{Implicational  Propositional Calculus} 

\medbreak 

Throughout, we deal with the purely Implicational Propositional Calculus (IPC): this has the conditional ($\supset$) as its only connective and modus ponens (or MP) as its only inference rule and rests on the following axiom schemes: \par
(${\rm IPC}_1$) \; \; $X \supset (Y \supset X)$\par 
(${\rm IPC}_2$) \; \; $[X \supset (Y \supset Z)] \supset [(X \supset Y) \supset (X \supset Z)]$ \par 
(Peirce) \;  $[(X \supset Y) \supset X] \supset X$. \\
The first and second secure validity of $Z \supset Z$ as a theorem scheme, of the Deduction Theorem (DT) as a derived inference rule and of Hypothetical Syllogism (HS) as a special consequence; we shall use all of these freely, often without comment. We shall write $L$ for the set comprising all (well-formed) IPC formulas. We write $X \equiv Y$ to assert that IPC formulas $X$ and $Y$ are syntactically equivalent in the sense that both $X \vdash Y$ and $Y \vdash X$.

\medbreak 

Although IPC lacks negation, a partial substitute may be introduced as follows. Fix an IPC formula $Q \in L$: when $Z \in L$ is any IPC formula, write 
$$Q Z := Q(Z) := Z \supset Q$$
with the understanding that $QQ Z = (Z \supset Q) \supset Q$ and so forth. The following theorem is taken directly from Exercise 6.3 in [3]; as an exercise, the (omitted) proof offers a good introduction to IPC. 

\medbreak 

\begin{theorem} \label{Robbin} 
Each of the following is an IPC theorem scheme: \par
{\rm (1)} $(X \supset Y) \supset [(Y \supset Z) \supset (X \supset Z)]$ \par 
{\rm (2)} $(X \supset Y) \supset (QY \supset QX)$ \par 
{\rm (3)} $Z \supset QQZ$ \par
{\rm (4)} $QQQZ \supset QZ$ \par 
{\rm (5)} $QQY \supset QQ(X \supset Y)$ \par 
{\rm (6)} $QQX \supset [QY \supset Q(X \supset Y)]$ \par 
{\rm (7)} $QX \supset QQ(X \supset Y)$ \par 
{\rm (8)} $(QX \supset Y) \supset [(QQX \supset Y) \supset QQY].$ 
\end{theorem} 

\medbreak 

The partial resemblance of $QZ$ to a negation of $Z$ is manifest in this theorem; it is interesting to trace the resemblance in subsequent theorems and their proofs. Incidentally, it is noted in [3] that part (7) here involves the Peirce axiom scheme. 

\medbreak 

We shall require several further IPC theorem schemes and related results. No claim is laid to the most expeditious route possible. Indeed IPC is complete: see [3] for relevant exercises and [5] for a proof; this means that many of our results succumb to elementary semantic confirmation of the truth-table variety. We deliberately proceed along syntactic lines, not least because this approach allows us to bring out the importance of the Peirce axiom scheme. Some of our results do not bear directly on our final theorem, but they are included for their independent interest.  

\medbreak 

Our first step in this direction is as follows. 

\medbreak 

\begin{theorem} \label{double}
If $Z$ is an IPC formula then $QQZ \vdash Z$ precisely when $Q \vdash Z$. 
\end{theorem} 

\begin{proof} 
In the one direction, an instance of ${\rm IPC}_1$ yields $Q \vdash QZ \supset Q = QQZ$ so that if $QQZ \vdash Z$ then $Q \vdash Z$ follows. In the opposite direction, let $Q \vdash Z$ so that $\vdash Q \supset Z$: the assumption $QZ \supset Q$ yields $QZ \supset Z$ (by HS) and then the instance $(QZ \supset Z) \supset Z$ of the Peirce axiom scheme yields $Z$ (by MP); accordingly, if $Q \vdash Z$ then $QQZ \vdash Z$. 
\end{proof} 

\medbreak 

Recall from Theorem \ref{Robbin} part (3) that $Z \vdash QQZ$ in any case; it follows that $QQZ \equiv Z$ precisely when $Q \vdash Z$. 

\begin{theorem} \label{extra}
Each of the following is an IPC theorem scheme: \par
{\rm (1)} $Q(X \supset Y) \supset QQX$ \par 
{\rm (2)} $Q(X \supset Y) \supset QY.$
\end{theorem} 

\begin{proof} 
Our offering direct arguments would sabotage the implied exercise in Theorem \ref{Robbin}. We merely note that parts (2), (3) and (4) of Theorem \ref{Robbin} mediate between parts (7)(5) of Theorem \ref{Robbin} and parts (1)(2) of the present theorem. \end{proof} 

\medbreak 

The following equivalence is a partial version of the `law of contraposition'. 

\medbreak 

\begin{theorem} \label{contra} 
$QQ (X \supset Y) \equiv QY \supset QX.$
\end{theorem} 

\begin{proof} 
$QQ (X \supset Y) \vdash QY \supset QX$: To see this, assume $QQ(X \supset Y)$, $QY$ and $X$. In turn, we deduce $QQX$ (by part (3) of Theorem \ref{Robbin}), $QY \supset Q(X \supset Y)$ (by MP and part (6) of Theorem \ref{Robbin}), $Q(X \supset Y)$ (by MP) and $Q$ (by MP). Thus 
$$QQ (X \supset Y), QY, X \vdash Q$$
and so two applications of DT yield 
$$QQ(X \supset Y) \vdash QY \supset QX.$$
\noindent
$QY \supset QX \vdash QQ(X \supset Y)$: To see this, assume $QY \supset QX$ and $Q(X \supset Y)$. We deduce in turn $QQX (= QX \supset Q)$ (by part (1) of Theorem \ref{extra}), $QY$ (by part (2) of Theorem \ref{extra}), $QX$ (by MP) and $Q$ (by MP). Thus 
$$QY \supset QX, Q(X \supset Y) \vdash Q$$
and so an application of DT yields 
$$QY \supset QX \vdash QQ(X \supset Y).$$
\end{proof} 

\medbreak 

The next result is a partial version of `denial of the antecedent'. 

\medbreak 

\begin{theorem} \label{neg}
If $X$ and $B$ are IPC formulas then:\par 
(1) $QX \vdash X \supset QB$; \par 
(2) $X \vdash QX \supset QB.$
\end{theorem} 

\begin{proof} 
Assume $QX = X \supset Q$ and $X$: by MP we deduce $Q$; by the instance $Q \supset (B \supset Q)$ of ${\rm IPC}_1$ and MP we deduce $B \supset Q = QB$. This proves the deduction 
$$QX, X \vdash QB$$
from which (1) and (2) follow by separate applications of DT. 
\end{proof} 

\medbreak 

\medbreak 

Of course, (1) improves (2) to the statement $QQX \vdash QX \supset QB$. 

\medbreak 

Additional evidence for the action of $Q$ as a partial negation is provided by the next result. 

\medbreak 

\begin{theorem} 
If $A$ and $B$ are IPC formulas then: \par
(1) $QA \supset B, A \supset B \vdash QQB;$ \par 
(2) $A \supset QB, A \supset B \vdash QA.$
\end{theorem} 

\begin{proof} 
For (1) assume $QA \supset B, A \supset B$ and $B \supset Q$: there follow $A \supset Q = QA$ (by HS), $B$ (by MP), $Q$ (by MP); now apply DT. For (2) assume $A \supset QB, A \supset B$ and $A$: two separate applications of MP yield $QB = B \supset Q$ and $B$ whence a third yields $Q$; again apply DT. 
\end{proof} 

\medbreak 

We introduce disjunction into IPC as an abbreviation: explicitly, when $X$ and $Y$ are IPC formulas we define 
$$X \vee Y := (X \supset Y) \supset Y.$$
This derived connective has the properties expected of it. Among the most fundamental are those expressed in the following result. 

\medbreak 

\begin{theorem} \label{sup}
Let $X, Y, Z$ be IPC formulas. Then: \par 
(1) $X \vdash X \vee Y$ and $Y \vdash X \vee Y$; \par 
(2) if $X \vdash Z$ and $Y \vdash Z$ then $X \vee Y \vdash Z$. 
\end{theorem} 

\begin{proof} 
Part (1) is Theorem 2 in [4]: its first assertion follows from the assumptions $X$ and $X \supset Y$ by MP and then DT; its second assertion follows by MP from an instance of ${\rm IPC}_1$. Part (2) is Theorem 3 in [4] and is more substantial: part (7) of Theorem \ref{Robbin} gives 
$$X \supset Z \vdash ((X \supset Y) \supset Z) \supset Z$$
while HS gives 
$$Y \supset Z, (X \supset Y) \supset Y \vdash (X \supset Y) \supset Z$$
whence 
$$X \supset Z, Y \supset Z, (X \supset Y) \supset Y \vdash Z$$
by MP and therefore DT yields 
$$X \supset Z, Y \supset Z \vdash (X \vee Y) \supset Z.$$
\end{proof} 

\medbreak 

We should point out here the r\^ole played by the Peirce axiom scheme, which enters in the form of Theorem \ref{Robbin} part (7); see [4] for more on this. Observe that as a consequence of this theorem, $\vee$ is symmetric in the sense $Y \vee X \equiv X \vee Y$. 

\medbreak 

\begin{theorem} \label{dM}
If $X$ and $Y$ are IPC formulas then 
$$QX \vee QY \equiv Y \supset (X \supset Q).$$
\end{theorem} 

\begin{proof} 
$QX \vee QY \equiv Y \supset (X \supset Q)$: On the one hand, $QX \vdash Y \supset QX$ follows from an instance of ${\rm IPC}_1$; on the other hand, $QY, Y \vdash Q$ so that $QY, Y \vdash QX$ by ${\rm IPC}_1$ and therefore $QY \vdash Y \supset QX$. An application of Theorem \ref{sup} part (2) ends the argument. \par 
$Y \supset (X \supset Q) \vdash QX \vee QY$: Assume $Y \supset (X \supset Q), (X \supset Q) \supset (Y \supset Q)$ and $Y$; three applications of MP yield $Q$ whereupon two successive applications of DT yield 
$$Y \supset (X \supset Q), QX \supset QY \vdash QY$$
and then 
$$Y \supset (X \supset Q) \vdash (QX \supset QY) \supset QY = QX \vee QY.$$

\end{proof} 

\medbreak 

The preceding partial `de Morgan law'  is an instance of Theorem 6 in [5]. 

\medbreak 

\begin{theorem} \label{closed}
If $X$ and $B$ are IPC formulas then there exist IPC formulas $C$ and $D$ such that: \par
(1) $X \supset QB \equiv QC;$ \par 
(2) $X \vee QB \equiv QD$. 
\end{theorem} 

\begin{proof} 
Part (1): As regards the conditional, $Q \vdash B \supset Q$ and $QB \vdash X \supset QB$ by instances of ${\rm IPC}_1$ so that $Q \vdash X \supset QB$; now Theorem \ref{double} tells us that $X \supset QB \equiv QC$ with $C = Q(X \supset QB).$
Part (2) follows as a consequence: as $X \vee QB = (X \supset QB) \supset QB$ we may take $D = Q(X \vee QB).$ 
\end{proof} 

\medbreak 

Notice that if $Z$ denotes either $X \supset QB$ or $X \vee QB$ then $QQZ \equiv Z$  by Theorem \ref{double}. 

\medbreak 

\begin{theorem} \label{D}
If $X, B, Z$ are IPC theorems and $Y = QB$ then 
$$Q(Q(X \vee Y) \vee QZ) \vdash QQX \vee Q(QY \vee QZ).$$
\end{theorem} 

\begin{proof} 
 We begin the proof with a sequence of claims.\par
{\it Claim} 1:  $Q(X \vee Y) \supset QZ, Z \vdash X \vee Y.$ [Theorem \ref{contra} gives $Q(X \vee Y) \supset QZ \equiv QQ(Z \supset (X \vee Y))$ while $QQ(Z \supset (X \vee Y)) \equiv Z \supset (X \vee Y)$ by Theorem \ref{closed}; an application of MP concludes the argument.] \par 
{\it Claim} 2: $QX, QY, X \vee Y \vdash Q.$ [Theorem \ref{neg} yields $QX \vdash X \supset Y$; as $X \vee Y = (X \supset Y) \supset Y$ it follows that $QX, X \vee Y \vdash Y$ by MP; as $QY = Y \supset Q$ we conclude that $QX, QY, X \vee Y \vdash Q$ by an application of  MP.]\par 
{\it Claim} 3: $QY \vee QZ, QX \vdash Q(X \vee Y) \vee QZ.$ [Assume $QY \vee QZ, QX, Q(X \vee Y) \supset QZ$ and $Z$. We deduce $Z \supset QY$ (directly from $QY \vee QZ$ by Theorem \ref{dM}), $QY$ (by MP), $X \vee Y$ (by Claim 1) and $Q$ (by Claim 2). Thus 
$$QY \vee QZ, QX, Q(X \vee Y) \supset QZ, Z \vdash Q$$
and so DT yields 
$$QY \vee QZ, QX, Q(X \vee Y) \supset QZ \vdash Z \supset Q = QZ$$
whence 
$$QY \vee QZ, QX \vdash (Q(X \vee Y) \supset QZ) \supset QZ = Q(X \vee Y) \vee QZ$$
by a further application of DT.] \par 
We now complete the proof as follows. Assume $Q(Q(X \vee Y) \vee QZ), QY \vee QZ$ and $QX$. By Claim 3 we deduce $Q(X \vee Y) \vee QZ$ and then by MP we deduce $Q$. This establishes the deduction 
$$Q(Q(X \vee Y) \vee QZ), QY \vee QZ, QX \vdash Q$$ 
from which 
$$Q(Q(X \vee Y) \vee QZ), QY \vee QZ \vdash QX \supset Q$$
follows by DT and 
$$Q(Q(X \vee Y) \vee QZ) \vdash (QY \vee QZ) \supset (QX \supset Q)$$
follows likewise. Finally, Theorem \ref{dM} justifies the equivalence  
$$ (QY \vee QZ) \supset (QX \supset Q) \equiv QQX \vee Q(QY \vee QZ). $$
\end{proof} 

\medbreak 

As a special case, if also $X = QA$ then $QQX \equiv X$ and therefore
$$Q(Q(X \vee Y) \vee QZ) \vdash X \vee Q(QY \vee QZ).$$

\medbreak 

\section{Lindenbaum-Tarski and  Gr\"atzer-Schmidt}

\medbreak 

The Lindenbaum-Tarski algebra of the Implicational Propositional Calculus results from the identification of syntactically equivalent IPC formulas. Explicitly, recall that the set $L$ comprising all (well-formed) IPC formulas is equipped with an equivalence relation $\equiv$ defined by the rule that $X \equiv Y$ precisely when both $X \vdash Y$ and $Y \vdash X$. As a point of notation, we shall typically name $\equiv$-classes and their representatives by lower-case and upper-case versions of the same letter: thus, if $Z$ is an IPC formula then $z = [Z]$ is its $\equiv$-class; conversely, if $w$ is a $\equiv$-class then $W$ will denote an IPC formula that represents it. The Lindenbaum-Tarski algebra $\LL = L / \equiv$ is the set comprising all such equivalence classes.

\medbreak 

The conditional $\supset$ descends to define on $\LL$ an operation for which we use the same symbol: thus, we define
$$[X] \supset [Y] := [X \supset Y].$$
To see that this operation is well-defined, let $X_0 \equiv X$ and $Y \equiv Y_0$: from $X_0 \vdash X$ and $Y \vdash Y_0$ there follow $\vdash X_0 \supset X$ and $\vdash Y \supset Y_0$ whence two applications of HS yield the deduction $X \supset Y \vdash X_0 \supset Y_0$; likewise, $X_0 \supset Y_0 \vdash X \supset Y$. 

\medbreak 

This operation on $\LL$ has several algebraic properties of interest. Among them is the self-distributive law: if $x, y, z \in \LL$ then 
$$x \supset (y \supset z) = (x \supset y) \supset (x \supset z).$$
To see this, let $x = [X], y = [Y], z = [Z]$. On the one hand
$$X \supset (Y \supset Z) \vdash (X \supset Y) \supset (X \supset Z)$$
follows from an instance of ${\rm IPC}_2$. On the other hand, assume $(X \supset Y) \supset (X \supset Z), X$ and $Y$: in turn there follow $X \supset Y$ (by ${\rm IPC}_1$ and MP), $X \supset Z$ (by MP), $Z$ (by MP); thus DT yields $(X \supset Y) \supset (X \supset Z), X \vdash Y \supset Z$ and a final application of DT yields 
$$(X \supset Y) \supset (X \supset Z) \vdash X \supset (Y \supset Z).$$
Also of interest is a commutative law for stacked antecedents: 
$$x \supset (y \supset z) = y \supset (x \supset z).$$
To see this by symmetry, note that MP twice yields 
$$X \supset (Y \supset Z), Y, X \vdash Z$$
whence DT twice yields 
$$X \supset (Y \supset Z) \vdash Y \supset (X \supset Z).$$

\medbreak 

The Lindenbaum-Tarski algebra $\LL$ also carries a disjunctive operation $\vee$ well-defined by iterating the conditional: 
$$[X] \vee [Y] := ([X] \supset [Y]) \supset [Y]$$ 
so that of course 
$$[X] \vee [Y] = [X \vee Y].$$

\medbreak 

A partial order $\leqslant$ is well-defined on $\LL$ by the declaration 
$$[X] \leqslant [Y] \Leftrightarrow X \vdash Y.$$
The poset $(\LL, \leqslant)$ is plainly topped: its unit $\bf 1$ is precisely the $\equiv$-class comprising all IPC theorems; indeed, if $T$ is a theorem and $Z$ a formula then $Z \vdash T$ so that $[Z] \leqslant [T]$. More is true, as noted in [4]: the poset $(\LL, \leqslant)$ is actually a semilattice, pairwise suprema being given by the disjunctive operation $\vee$ introduced above, so that if $x, y \in \LL$ then 
$$\sup \{ x, y \} = x \vee y;$$
this is an immediate consequence of Theorem \ref{sup} and depends crucially on the Peirce axiom scheme. In short, $\LL$ is a topped join-semilattice. 
\bigbreak

Now, choose and fix an arbitrary element $q = [Q]$ of $\LL$. Consider the map 
$$(\bullet)^q : \LL \ra \LL : z \mapsto z^q := z \supset q$$
so that if $z = [Z]$ then $z^q = [QZ]$. In the special case that $Q$ is a theorem (so that $q = {\bf 1}$ is the unit) this map takes ${\bf 1}$ as its constant value; our interest lies largely in the complementary case. 

\medbreak 

The map $(\bullet)^q$ is order-reversing. 

\medbreak 

\begin{theorem} \label{anti}
If $x, y, z \in \LL$ then $x \leqslant y \Rightarrow y^q \leqslant x^q$ and $z \leqslant z^{qq}$. 
\end{theorem} 

\begin{proof} 
If $X \vdash Y$ then $\vdash X \supset Y$ whence Part (2) of Theorem \ref{Robbin} yields $\vdash QY \supset QX$ and therefore $QY \vdash QX$. Part (3) of Theorem \ref{Robbin} similarly tells us that $Z \vdash QQZ$. 
\end{proof} 

\medbreak 

We write $\LL^q$ for the image of $(\bullet)^q$: thus
$$\LL^q = \{ z \supset q : z \in \LL \} \subseteq \LL.$$
Also associated to $q$ is its up-set, defined by 
$$\uparrow q = \{ w \in \LL : q \leqslant w \} \subseteq \LL.$$
In fact these two sets, the one defined algebraically and the other order-theoretically, coincide. 

\medbreak 

\begin{theorem} \label{up}
$\LL^q = \; \uparrow q.$
\end{theorem} 

\begin{proof} 
The inclusion $\LL^q \subseteq \; \uparrow q$ follows by virtue of the deduction $Q \vdash Z \supset Q$ from an instance of ${\rm IPC}_1$. The inclusion $\uparrow q \subseteq \LL^q$ follows from Theorem \ref{double}: if $Q \vdash W$ then $W \equiv QQW = QW \supset Q$. 
\end{proof}

\medbreak 

In a similar vein, we note that 
$$\LL^q = \{ z \in \LL : z^{qq} = z \}$$
as a consequence of Theorem \ref{Robbin} parts (3) and (4) for instance. Accordingly, the restriction of $(\bullet)^q$ to $\LL^q$ is an involution: if $w = z^q \in \LL^q$ then $w^{qq} = z^{qqq} = z^q = w$. 

\medbreak 

As an up-set, $\LL^q$ is closed under going up: if $\LL^q \ni x \leqslant y \in \LL$ then $y \in \LL^q$. The behaviour of $\LL^q$ relative to $\supset$ and $\vee$ is similar. 

\medbreak 

\begin{theorem} \label{closure}
If $x \in \LL$ and $y \in \LL^q$ then $x \supset y \in \LL^q$ and $x \vee y \in \LL^q$. 
\end{theorem} 

\begin{proof} 
By hypothesis, $y = [Y] = [QB]$ for some $B \in L$: Theorem \ref{neg} provides $c = [C]$ and $d = [D]$ such that $x \supset y = [QC] = c^q \in \LL^q$ and $x \vee y = [QD] = d^q \in \LL^q$. 
\end{proof} 

\medbreak 

In particular, $\LL^q$ is closed under the operations $\supset$ and $\vee$. More particularly still, the poset $\LL^q$ is itself a join-semilattice. 

\medbreak 

The map $(\bullet)^q$ facilitates our defining on $\LL$ an operation that shares some properties with conjunction or meet. To be explicit, we define 
$$\stackrel{q}{\wedge} \; : \LL \times \LL \ra \LL^q$$
by the requirement that if $x \in \LL$ and $y \in \LL$ then 
$$x \wedge y := (x^q \vee y^q)^q$$
so that in terms of representative IPC formulas
$$[X] \wedge [Y] = [Q(QX \vee QY)].$$
As indicated, we shall henceforth dispense with the superscript on $\wedge$ that signifies its dependence on $q$. 

\medbreak 

\begin{theorem} 
Let $x, y \in \LL^q$ and $z \in \LL$. Then: \par 
(1) $x \wedge y \leqslant x$ and $x \wedge y \leqslant y$; \par 
(2) if $z \leqslant x$ and $z \leqslant y$ then $z \leqslant x \wedge y.$ 
\end{theorem} 

\begin{proof} 
Throughout, we recall that $\vee$ furnishes pairwise suprema, that $(\bullet)^q$ reverses order (as in Theorem \ref{anti}) and that if $w \in \LL^q$ then $w^{qq} = w$ (as noted after Theorem \ref{up}). For (1) note that $x^q \leqslant x^q \vee y^q$ so that  $(x^q \vee y^q)^q \leqslant x^{qq} = x$; thus $x \wedge y \leqslant x$ while $x \wedge y \leqslant y$ by a similar argument. For (2) note that  $x^q \leqslant z^q$ and $y^q \leqslant z^q$ whence $x^q \vee y^q \leqslant z^q$ and therefore $z \leqslant z^{qq} \leqslant  (x^q \vee y^q)^q = x \wedge y$. 
\end{proof} 

\medbreak 

Now $\LL^q$ is actually a lattice: before this theorem, $\LL^q$ was already a join-semilattice; after this theorem, $\LL^q$ is also a meet-semilattice. In fact, the lattice $\LL^q$ is bounded: its top element is the $\equiv$-class ${\bf 1} = [Q \supset Q]$ of all theorems, as noted previously; its bottom element is the $\equiv$-class ${\bf 0} = q = [Q] = [(Q \supset Q) \supset Q]$ as is evident from Theorem \ref{up}. 

\medbreak 

\begin{theorem} \label{ps}
If $z \in \LL$ then $z^q \vee z = {\bf 1}$ and $z^q \wedge z = q$.  
\end{theorem} 

\begin{proof} 
If $Z \in L$ then $QZ \vee Z = ((Z \supset Q) \supset Z) \supset Z$ is an instance of the Peirce axiom scheme, whence $z^q \vee z = [QZ \vee Z] = {\bf 1}$; this proves the first identity. The second identity follows: as $z^{qq} \vee z^q = {\bf 1}$ so $z^q \wedge z = (z^{qq} \vee z^q)^q = {\bf 1}^q = {\bf 1} \supset q = q$. 
\end{proof}

\medbreak 

This reformulation $z^q \vee z = {\bf 1}$ of the Peirce scheme amounts to a partial `law of the excluded middle'. 

\medbreak 

In particular, the bounded lattice $\LL^q$ is complemented. 

\medbreak 

\begin{theorem} \label{dist}
The complemented bounded lattice $\LL^q$ is distributive. 
\end{theorem} 

\begin{proof} 
According to Lemma 4.10 in [1] we need only show that if $x, y, z \in \LL^q$ then 
$$(x \vee y) \wedge z \leqslant x \vee (y \wedge z).$$
Theorem \ref{D} was prepared for this purpose: indeed, $X \equiv QQX$ while 
$$(x \vee y) \wedge z = [X \vee Y] \wedge [Z] = [Q(Q(X \vee Y) \vee QZ)]$$
and
$$x \vee (y \wedge z) = [X] \vee ([Y] \wedge [Z]) = [X \vee Q(QY \vee QZ)].$$
\end{proof} 

\medbreak 

According to a theorem of Huntington [2]  distributivity of the complemented lattice $\LL^q$ also follows from the fact that if the elements $z$ and $w$ of $\LL^q$ satisfy $z \wedge w = {\bf 0}$ then $w \leqslant z^q$. In fact, a stronger statement is true: if the elements $z$ and $w$ of $\LL$ itself satisfy $z \wedge w = q$ then $w \leqslant z^q$. To see this, note that from $q = z \wedge w = (z^q \vee w^q)^q$ there follows ${\bf 1} = q^q = (z^q \vee w^q)^{qq} = z^q \vee w^q$: in terms of representative elements, $\vdash QZ \vee QW$ so that $\vdash W \supset QZ$ by Theorem \ref{dM} and therefore $W \vdash QZ$; hence $w \leqslant z^q$ as claimed. 

\medbreak 

Thus, $\LL^q$ is actually a Boolean lattice. We summarize our findings as follows. 

\medbreak 

\begin{theorem} \label{main}
The Lindenbaum-Tarski algebra $\LL$ of IPC is a topped join-semilattice. If $q \in \LL$ is arbitrary then the inherited partial order makes $\LL^q = \{ z^q : z \in \LL \}$ into a Boolean lattice, such that if $x, y \in \LL^q$ then $\sup \{ x, y \} = x \vee y$ and $\inf \{x, y\} = (x^q \vee y^q)^q.$
\end{theorem} 

\medbreak 

Our statement of this result is intentionally reminiscent of a theorem due to Gr\"atzer and Schmidt [1]. To formulate the Gr\"atzer-Schmidt theorem, let $\MM$ be a semilattice: either a join-semilattice in which $\vee$ denotes pairwise supremum (and assume a unit ${\bf 1}$) or a meet-semilattice in which $\wedge$ denotes pairwise infimum (and assume a zero ${\bf 0}$). Let $\MM$ be correspondingly pseudo-complemented: in the $\vee$ case, the $\vee$-pseudocomplement $a^*$ of $a \in \MM$ satisfies $a \vee a^* = {\bf 1}$ and if $x \in \MM$ then $a \vee x = {\bf 1} \Rightarrow a^* \leqslant x$; in the $\wedge$ case, the  $\wedge$-pseudocomplement $a^*$ of $a \in \MM$ satisfies $a^* \wedge a = {\bf 0}$ and if $x \in \MM$ then $x \wedge a = {\bf 0} \Rightarrow x \leqslant a^*$. We should mention that traditional terminology is less even-handed: $\wedge$-pseudocomplements are simply called pseudocomplements; $\vee$-pseudocomplements are then called dual pseudocomplements. Perhaps it would be too much to suggest $a^*$ for the $\vee$-pseudocomplement of $a$ and $a_*$ for the $\wedge$-pseudocomplement of $a$. 

\medbreak 

In these terms, the Gratzer-Schmidt theorem is a dual pair: [1] presents the version for a meet-semilattice; here we state the version for a join-semilattice. 

\medbreak 
\noindent
{\bf Theorem} (Gr\"atzer-Schmidt). {\it Let $\MM$ be a pseudocomplemented semilattice with join $\vee$ and let $S(\MM) = \{ a^* : a \in \MM \}$ be its skeleton. The partial order on $\MM$ makes $S(\MM)$ into a Boolean lattice. For $a, b \in S(\MM)$ the join is $a \vee b$ and the meet is $(a^* \vee b^*)^*$. }

\medbreak 

The parallel between Theorem \ref{main} and the Gr\"atzer-Schmidt Theorem is clear but not exact. Our unary operation $(\bullet)^q$ is not a $\vee$-pseudocomplementation. Certainly, if $z \in \LL$ then $z^q \vee z = {\bf 1}$: this is essentially a reformulation of the Peirce scheme. However, it is not generally the case that if also $w \in \LL$ then $z \vee w = {\bf 1} \Rightarrow z^q \leqslant w$. Let $Q = Q_0 \supset Q_1$ where $Q_0$ is not a theorem; let $Z = Q$ and $W = Q_0$. In this case, $Z \vee W = ((Q_0 \supset Q_1) \supset Q_0) \supset Q_0$ is a theorem (according to Peirce) but $QZ \supset W = (Q \supset Q) \supset Q_0$ is not (for $Q \supset Q$ is a theorem but $Q_0$ is not); so $z \vee w = {\bf 1}$ is satisfied but $z^q \leqslant w$ is not. 

\medbreak 

The structures discussed here naturally depend on the choice of $q = [Q]$; investigation of this dependence is among topics reserved for a future publication. 

\bigbreak

\begin{center} 
{\small R}{\footnotesize EFERENCES}
\end{center} 
\medbreak

[1] G. Gr\"atzer, {\it Lattice Theory - First Concepts and Distributive Lattices}, W.H. Freeman (1971); Dover Publications (2009). 

[2] E. V. Huntington, {\it Sets of independent postulates for the algebra of logic}, Trans. Amer. Math. Soc. 5: 288-309 (1904). 

[3] J. W. Robbin, {\it Mathematical Logic - A First Course}, W.A. Benjamin (1969); Dover Publications (2006).

[4] P. L. Robinson, {\it The Peirce axiom scheme and suprema}, arXiv 1511.07074 (2015).

[5] P. L. Robinson, {\it $Q$-tableaux for Implicational Propositional Calculus}, arXiv 1512.03525 (2015). 

\medbreak

\end{document}